%% file: LSI_on_Free_Groups.tex
\documentclass[a4paper]{amsart}
\usepackage{mathmacro}

\usepackage[maxalphanames=99,url=false,maxbibnames=99,giveninits=true,style=alphabetic]{biblatex}
\AtEveryBibitem{
  \ifentrytype{article}{\clearfield{doi}\clearfield{issn}\clearfield{isbn}}{}
  \ifentrytype{incollection}{\clearfield{doi}\clearfield{issn}\clearfield{isbn}}{}
  \ifentrytype{book}{\clearfield{doi}\clearfield{issn}\clearfield{isbn}}{}
}

\usepackage{etoolbox}

\makeatletter

\@ifpackageloaded{hyperref}{
  \@ifpackageloaded{biblatex}{
    \DeclareRobustCommand{\SkipTocEntry}[5]{}
  }
}{}

\let\origsubsection\subsection
\renewcommand{\subsection}{
  \@ifstar{\subsection@star}{\origsubsection}%
}
\newcommand{\subsection@star}[1]{
  \addtocontents{toc}{\protect\SkipTocEntry}
  \origsubsection*{#1}
}

\let\origsection\section
\renewcommand{\section}{
  \@ifstar{\section@star}{\origsection}%
}
\newcommand{\section@star}[1]{
  \addtocontents{toc}{\protect\SkipTocEntry}
  \origsection*{#1}
}

\newcommand{\addresseshere}{
  \enddoc@text\let\enddoc@text\relax
}

\makeatother

\title[Log--Sobolev inequalities on free groups]{Sharp logarithmic Sobolev inequalities on free groups by variational comparison}
\date{\today}

\author{Gan Yao}

\address[Gan Yao]{Institute for Advanced Study in Mathematics, Harbin Institute of Technology,  Harbin 150001, China.}
\email{gan.yao.math@gmail.com}

\begin{document}
\begin{abstract}
We give a variational comparison proof of the sharp logarithmic Sobolev inequality with constant $2$ for the word-length Poisson semigroup on free group von Neumann algebras. Our approach extends the Lagrange method for cyclic groups in~\cite{yao2026completesolutionoptimalhypercontractivity} to a framework in finite tracial von Neumann algebras. The inequality was recently established by Xie and Zhang~\cite{xiezhang2026hypercontractivityfreegroup} using a completely different approach. 
\end{abstract}
\maketitle
\tableofcontents
\section{Introduction}\label{sec:introduction}

Let $\mathbb F_n$ be the free group on $n$ generators, where $n\in\mathbb N=\{1,2,\ldots\}$, and let $\mathcal M=\mathcal L(\mathbb F_n)$ be its group von Neumann algebra with canonical normalized trace $\tau$. We write $\lambda_g$ for the unitary associated with $g\in\mathbb F_n$ in the left regular representation, and $|g|$ for its reduced word length. The word-length generator and the associated Poisson semigroup are given by
\[
 A\lambda_g=|g|\lambda_g,\qquad
 P_t\lambda_g=e^{-t|g|}\lambda_g.
\]
For $x\in L_2(\mathcal M,\tau)_+$, define
\[
 \Ent_\tau(x^2)=\tau(x^2\log x^2)
       -\norm x_2^2\log\norm x_2^2,
\]
with $0\log0=0$, allowing the value $+\infty$. We prove the following sharp logarithmic Sobolev inequality.

\begin{theorem}\label{thm:LSI}
Let $n\in\mathbb N$. Then
\begin{equation}\label{eq:LSI}
 \Ent_\tau(x^2)\le2\norm{A^{1/2}x}_2^2,
 \qquad x\in\Dom(A^{1/2})_+.
\end{equation}
The constant $2$ is optimal.
\end{theorem}
Theorem~\ref{thm:LSI} was recently established by Xie and Zhang~\cite{xiezhang2026hypercontractivityfreegroup}. We give a proof by a \textbf{variational comparison method}, extending the Lagrange approach to cyclic groups in~\cite{yao2026completesolutionoptimalhypercontractivity}. Following the variational approach to logarithmic Sobolev inequalities (LSI) in~\cite{MR593787,MR674060}, we first relate failure of \eqref{eq:LSI} to a nonlinear Lagrange equation. If \eqref{eq:LSI} fails, a minimizer $x_0$ satisfies
\[
      F(x_0)=2\norm{A^{1/2}x_0}_2^2-\Ent_\tau(x_0^2)=m<0.
\]
By the Markov property, $|x_0|$ is also a minimizer. The Lagrange multiplier theorem and a rescaling then give
\[
 x=e^{m/2}|x_0|\in\Dom(A)_+,\qquad
 Ax=x\log x,\qquad 0<\norm x_2=e^{m/2}<1.
\]
Conversely, any such solution violates the LSI.

The comparison step rules out these Lagrange solutions using \textbf{noncommutative derivations}. Let $\delta$ be a closed derivation, and let $L=\delta^*\delta$ be its associated Laplacian. Define the comparison functional $K_L(x)$ and the endpoint gap $G_L(x)$ by
\begin{equation}\label{eq:KG}
 \begin{aligned}
 K_L(x)&=2\operatorname{Re}\inner{Lx,Ax}_2-\operatorname{Re}\inner{Lx,x^2}_2,\qquad x\in\Dom(A)_+\cap\Dom(L)_+,\\
 G_L(x)&=\operatorname{Re}\inner{Lx,x^2-2x\log x}_2,\qquad x\in\Dom(A)_+\cap\Dom(L)_+.
 \end{aligned}
\end{equation}
Their sum satisfies
\begin{equation}\label{eq:KGsum}
 K_L(x)+G_L(x)=2\operatorname{Re}\inner{Lx,Ax-x\log x}_2.
\end{equation}
The key idea is to choose a Laplacian $L$ whose domain contains the solutions under consideration and for which
\begin{equation}\label{eq:positive-comparison}
 K_L(x)+G_L(x)>0
 \qquad\text{for nonconstant }x\in\Dom(A)_+\cap\Dom(L)_+
 \text{ with }0<\norm x_2<1.
\end{equation}
A nonconstant Lagrange solution makes the left-hand side zero, so this comparison yields the required contradiction.\par

In the free group setting, we use the derivation associated with the word-length cocycle, for which $L=A$, and write $K=K_A$ and $G=G_A$. The joint spectral representation of $G(x)$ gives a lower bound in terms of the energy and a cubic moment. We control the cubic term in $K(x)$ by a length-weighted Fourier estimate. Its proof uses the cancellation estimate also employed in~\cite{xiezhang2026hypercontractivityfreegroup}, with a word-length weight adapted to the pairing $\inner{Au,u^2}_2$. A scalar comparison then gives \eqref{eq:positive-comparison}.

The article is organized as follows. Section~\ref{sec:preliminaries} fixes notation and recalls the analytic tools. Section~\ref{sec:framework} develops the variational criterion and the integral representation of $G_L$ in general tracial von Neumann algebras. Section~\ref{sec:free-group-settings} verifies their hypotheses for free groups and reduces Theorem~\ref{thm:LSI} to the cubic estimate. Section~\ref{sec:shell} proves that estimate. The extension to $\mathbb F_\infty$ is given in Remark~\ref{rem:variational-domain}.
\input{Preliminaries.tex}

\input{Framework_Euler.tex}
\input{Free_group_setting.tex}
\input{Shell_Estimation.tex}

\section*{Acknowledgments}
The author would like to thank Professor Quanhua Xu and Professor Simeng Wang for their patience and encouragement, as well as for their careful reading of the manuscript and many helpful discussions. The author is partially supported by the NSF of China (No. 12031004, No. W2441002, No. 12301161, No.12371138). The author acknowledges the use of ChatGPT-6 Astra.

\printbibliography
\addresseshere
\end{document}

%% file: Preliminaries.tex
\section{Preliminaries}\label{sec:preliminaries}
Throughout this section and the next, $(\mathcal M,\tau)$ is a finite von Neumann algebra equipped with a faithful normal tracial state.

\subsection*{Tracial von Neumann algebras and \texorpdfstring{$L_p$}{Lp} spaces}
For $1\le p<\infty$, define $L_p(\mathcal M,\tau)$ as the completion of $\mathcal M$ with respect to
\[
 \norm x_p=\tau(|x|^p)^{1/p},\qquad x\in\mathcal M,
\]
where $|x|=(x^*x)^{1/2}$. Its elements can be identified with possibly unbounded $\tau$-measurable operators affiliated with $\mathcal M$. We put $L_\infty(\mathcal M,\tau)=\mathcal M$ with the operator norm and abbreviate these spaces to $L_p$ when the algebra and trace are clear. The subscripts $\mathrm{sa}$ and $+$ denote the self-adjoint part and the positive cone, respectively. We use inner products that are linear in the second variable; in particular,
\[
 \inner{x,y}_2=\tau(x^*y),\qquad x,y\in L_2.
\]
The normalization of $\tau$ gives $\norm x_q\le\norm x_p$ for $1\le q\le p\le\infty$. We refer to \cite{MR1999201} for these facts and further background on noncommutative $L_p$ spaces.

For a discrete group $G$ with identity $e$, the left regular representation acts on $\ell_2(G)$ by $\lambda_g\delta_h=\delta_{gh}$, where $(\delta_h)_{h\in G}$ is the canonical orthonormal basis. The group von Neumann algebra is
\[
 \mathcal L(G)=\overline{\operatorname{span}}^{\,w^*}\{\lambda_g:g\in G\}
 \subset B(\ell_2(G)).
\]
Its canonical faithful normal tracial state is given by
\[
 \tau(x)=\inner{\delta_e,x\delta_e}_{\ell_2(G)},\qquad
 \tau(\lambda_g)=\begin{cases}1,&g=e,\\0,&g\ne e.\end{cases}
\]
For $f\in\ell_1(G)$, the series $\lambda(f)=\sum_{g\in G}f(g)\lambda_g$ converges in operator norm and satisfies $\tau(\lambda(f))=f(e)$. If $G$ is abelian, the Fourier transform identifies $\mathcal L(G)$ with $L_\infty(\widehat G)$ and its noncommutative $L_p$ spaces with the classical $L_p$ spaces on the compact Pontryagin dual $\widehat G$, equipped with normalized Haar measure.

\subsection*{Spectral calculus}
Let $S,T$ be self-adjoint operators on a Hilbert space $\mathcal K$, with spectral resolutions $E_S,E_T$. They are said to \emph{strongly commute} if every spectral projection of $S$ commutes with every spectral projection of $T$. In this case, there is a unique joint projection-valued measure $P$ on $\mathbb R^2$, characterized by $P(B\times C)=E_S(B)E_T(C)$ for Borel sets $B,C\subset\mathbb R$. The joint functional calculus is defined, for a Borel function $F:\mathbb R^2\to\mathbb C$, by
\[
 F(S,T)=\int_{\mathbb R^2}F(s,t)\,\mathrm dP(s,t).
\]
For $\xi\in\mathcal K$, the scalar measure $\mu_\xi(\Omega)=\inner{\xi,P(\Omega)\xi}_{\mathcal K}$ is positive and has mass $\norm\xi_{\mathcal K}^2$. For bounded Borel $F$,
\begin{equation}\label{eqn: joint spectral probability measure}
    \inner{\xi,F(S,T)\xi}_{\mathcal K}=\int F\,\mathrm d\mu_\xi,\qquad
 \norm{F(S,T)\xi}_{\mathcal K}^2=\int|F|^2\,\mathrm d\mu_\xi.
\end{equation}
For unbounded Borel $F$, finiteness of the second integral is equivalent to $\xi\in\Dom(F(S,T))$, and both identities remain valid on this domain. If $F$ is real and $\int|F|\,\mathrm d\mu_\xi<\infty$, the first identity also holds in the quadratic-form sense: its left-hand side is interpreted as
\[
 \lim_{m\to\infty}\inner{\xi,F_m(S,T)\xi}_{\mathcal K},\qquad
 F_m=\max\{-m,\min\{F,m\}\}.
\]
Dominated convergence identifies this limit with $\int F\,\mathrm d\mu_\xi$; square integrability is needed only to interpret $F(S,T)\xi$ as a vector. We refer to \cite{MR2953553} for further background on the spectral theory of commuting normal operators.

\subsection*{Markov semigroups and quadratic forms}
We recall the semigroup and quadratic-form framework in the tracial setting; see \cite{MR2463708}. Let $(P_t)_{t\ge0}$ be a semigroup of positive maps on $\mathcal M$ that is unital, trace-preserving and symmetric with respect to the $L_2$ inner product:
\[
 P_t\mathbf1=\mathbf1,\qquad \tau(P_ta)=\tau(a),\qquad
 \inner{P_tu,v}_2=\inner{u,P_tv}_2.
\]
These maps extend to contractions on $L_2$. We assume that the extended semigroup is strongly continuous and write $P_t=e^{-tA}$, where $A\ge0$ is self-adjoint. Thus $-A$ is its infinitesimal generator, and unitality gives $A\mathbf1=0$. The associated form domain and real bilinear form are
\[
 \mathcal Q=\Dom(A^{1/2}),\qquad
 \mathcal E(u,v)=\operatorname{Re}\inner{A^{1/2}u,A^{1/2}v}_2,\qquad
 \mathcal E(u)=\mathcal E(u,u).
\]
The space $\mathcal Q_{\mathrm{sa}}$ is a real Hilbert space with norm $\norm u_{\mathcal Q}=(\norm u_2^2+\mathcal E(u))^{1/2}$. Extending $\mathcal E$ to $+\infty$ outside $\mathcal Q$, the spectral theorem gives
\[
 \mathcal E(v)=\lim_{t\downarrow0}\frac1t\inner{v,(I-P_t)v}_2,
 \qquad v\in L_2.
\]
For $z=z^*\in\mathcal Q$, write $z=z_+-z_-$ for its decomposition into positive and negative parts. Positivity and symmetry of $P_t$ give
\[
 \inner{|z|,P_t|z|}_2-\inner{z,P_tz}_2
 =4\inner{z_+,P_tz_-}_2\ge0.
\]
Since $\norm{|z|}_2=\norm z_2$, the preceding limit formula yields
\begin{equation}\label{eq:modulus-form}
 |z|\in\mathcal Q,\qquad \mathcal E(|z|)\le\mathcal E(z).
\end{equation}
For $x\in\mathcal Q_{\mathrm{sa}}$ and $y=y^*\in L_2$, the condition $\mathcal E(x,h)=\operatorname{Re}\inner{y,h}_2$ for every $h\in\mathcal Q_{\mathrm{sa}}$ is equivalent to $x\in\Dom(A)$ and $Ax=y$. 

\subsection*{Fr\'echet differentiation and Lagrange multipliers}
We use the notions of differentiation from \cite[Sections~7.2--7.3]{MR238472}. Let $X$ be a real Banach space and $U\subset X$ be open. A functional $\Phi:U\to\mathbb R$ is Fr\'echet differentiable at $x\in U$ if there is a bounded real linear functional $D\Phi(x):X\to\mathbb R$ such that
\[
 \Phi(x+h)=\Phi(x)+D\Phi(x)[h]+o(\norm h_X)
 \qquad\text{as }\norm h_X\to0.
\]
The notation $D\Phi(x)[h]$ denotes the value of this linear functional at $h$. Fr\'echet differentiability at $x$ implies
\[
 D\Phi(x)[h]=\left.\frac{\mathrm d}{\mathrm dt}\Phi(x+th)\right|_{t=0},\qquad h\in X.
\]
The functional $\Phi$ is $C^1$ if it is Fr\'echet differentiable throughout $U$ and $x\mapsto D\Phi(x)$ is continuous in the norm of the dual space $X^*$.

We use the following scalar-constraint version of the Lagrange multiplier principle \cite[Section~9.3, Theorem~1]{MR238472}, which only requires differentiability of the objective at the extremum.
\begin{theorem}[Lagrange multipliers]\label{thm:lagrange-multipliers}
Let $X$ be a real Banach space, $U\subset X$ open, and $G:U\to\mathbb R$ continuously Fr\'echet differentiable. Suppose that $x_0\in U$ satisfies $G(x_0)=0$ and is a local minimizer or maximizer of $\Phi:U\to\mathbb R$ on $G^{-1}(\{0\})$. If $\Phi$ is Fr\'echet differentiable at $x_0$ and $DG(x_0)\ne0$, then there is a multiplier $\mu\in\mathbb R$ such that
\[
 D\Phi(x_0)[h]+\mu DG(x_0)[h]=0,\qquad h\in X.
\]
\end{theorem}

%% file: Framework_Euler.tex
\section{A variational framework in tracial von Neumann algebras}
\label{sec:framework}
In this section, we apply Theorem~\ref{thm:lagrange-multipliers} to a global minimizer of the homogeneous deficit on the $L_2$ unit sphere:
\[
 F_\rho(x)=2\rho\mathcal E(x)-\Ent_\tau(x^2),\qquad x\in\mathcal Q_{\mathrm{sa}}.
\]
We retain the form domain $\mathcal Q=\Dom(A^{1/2})$ and its norm $\norm{\cdot}_{\mathcal Q}=(\norm{\cdot}_2^2+\mathcal E(\cdot))^{1/2}$ from Section~\ref{sec:preliminaries}. We first establish the entropy differentiability needed for the multiplier theorem, then prove the variational criterion and the endpoint representation used to estimate $G_L$.
For $x=x^*\in L_2$, define
\[
 H(x)=\tau(x^2\log x^2),\qquad
 \Ent_\tau(x^2)=H(x)-\norm x_2^2\log\norm x_2^2.
\]
Since $s^2\log s^2\ge-e^{-1}$, the spectral integral defining $H$ is well defined with values in $[-e^{-1},+\infty]$. If $x\in L_p$ for some $p>2$, then $H(x)$ is finite.
\begin{lemma}[Entropy calculus on the form domain]\label{lem:entropy}
Assume that $\mathcal Q_{\mathrm{sa}}$, equipped with $\norm{\cdot}_{\mathcal Q}$, embeds continuously into $L_p$ for some $2<p<\infty$. Then $H:\mathcal Q_{\mathrm{sa}}\to\mathbb R$ is Fr\'echet differentiable, with
\begin{equation}\label{eq:trace-chain-Lp}
 DH(x)[h]=\tau\bigl((2x\log x^2+2x)h\bigr),\qquad x,h\in\mathcal Q_{\mathrm{sa}}.
\end{equation}
Here $DH(x)$ is the Fr\'echet derivative of $H$ at $x$, a bounded real linear functional on $\mathcal Q_{\mathrm{sa}}$, and $DH(x)[h]$ is its value at $h$.

Independently of the embedding assumption, $H$ is lower semicontinuous under strong $L_2$ convergence:
\begin{equation}\label{eq:entropy-lsc}
 x_j\to x\text{ in }L_2,\quad x_j=x_j^*,\ x=x^*
 \quad\Longrightarrow\quad H(x)\le\liminf_jH(x_j).
\end{equation}
\end{lemma}

\begin{proof}
Put $f(s)=s^2\log s^2$, with $f(0)=f'(0)=0$. For $x,y\in\mathcal Q_{\mathrm{sa}}$, let $\mu_{x,y}$ be the joint scalar measure of left multiplication by $x$ and right multiplication by $y$, with vector $\mathbf1$. Its marginals are the scalar spectral distributions of $x$ and $y$.

We first verify that $\tau(f(x))$ and $\tau(f'(x)h)$ are well defined for $x,h\in\mathcal Q_{\mathrm{sa}}$. The bound $|f(s)|+|f'(s)|^2\le C_p(1+|s|^p)$ gives $f(x),f(y)\in L_1$ and $f'(x),f'(y)\in L_2$. Thus $H(x)=\tau(f(x))$ is finite, and $h\mapsto\tau(f'(x)h)$ is a bounded real linear functional on $\mathcal Q_{\mathrm{sa}}$. Applying \eqref{eqn: joint spectral probability measure} in its integrable quadratic-form sense yields
\[
 H(y)-H(x)-\tau(f'(x)(y-x))
 =\int_{\mathbb R^2}\bigl(f(t)-f(s)-f'(s)(t-s)\bigr)\,\mathrm d\mu_{x,y}.
\]
For $r\ne0$, we have $|f''(r)|\le c_p|r|^{-1/2}(1+|r|^{p/4})$ for a constant $c_p$ depending only on $p$. Assume $s<t$. Since $f'$ is locally absolutely continuous, integration gives
\[
      |f'(t)-f'(s)|\leq c_p\int_{s}^{t}\frac{1+\abs{s}^{p/4}+\abs{t}^{p/4}}{\abs{r}^{\frac{1}{2}}}\dd r\le 2\sqrt{2}c_p|t-s|^{1/2}\bigl(1+|s|^{p/4}+|t|^{p/4}\bigr).
\]
The same bound holds with $s$ and $t$ interchanged. By the mean value theorem, there is a point $\xi$ between $s$ and $t$ such that
\[
      \abs{f(t)-f(s)-f'(s)(t-s)}\leq \abs{f'(\xi)-f'(s)}\abs{t-s}\leq 4\sqrt{2}c_p|t-s|^{3/2}\bigl(1+|s|^{p/4}+|t|^{p/4}\bigr).
\]
Applying H\"older's inequality with exponents $4/3$ and $4$, we obtain
\begin{equation}\label{eq:entropy-remainder}
 |H(y)-H(x)-\tau(f'(x)(y-x))|
 \lesssim \norm{y-x}_2^{3/2}
       \bigl(1+\norm x_p^p+\norm y_p^p\bigr)^{1/4}.
\end{equation}
With $y=x+h$, the continuous embedding makes the remainder $O(\norm h_{\mathcal Q}^{3/2})$ as $h\to0$ in $\mathcal Q_{\mathrm{sa}}$.
This proves \eqref{eq:trace-chain-Lp}.
Equation~\eqref{eq:entropy-remainder} also gives continuity of $H$ under $L_2$ convergence on subsets of $\mathcal Q_{\mathrm{sa}}$ bounded in $L_p$.

Finally, let $x_j=x_j^*\to x=x^*$ in $L_2$ and put $g_N(s)=\min\{f(s)+e^{-1},N\}$ for $N\ge1$. Each $g_N$ is nonnegative, bounded and Lipschitz. Applying the norm identity in \eqref{eqn: joint spectral probability measure} to left multiplication by $x_j$ and right multiplication by $x$, with vector $\mathbf1$, gives
\[
 \norm{g_N(x_j)-g_N(x)}_2
 \le\operatorname{Lip}(g_N)\norm{x_j-x}_2\longrightarrow0.
\]
Consequently, for every $N$,
\[
 \tau(g_N(x))=\lim_{j\to\infty}\tau(g_N(x_j))
 \le\liminf_{j\to\infty}\bigl(H(x_j)+e^{-1}\bigr).
\]
Since $g_N\uparrow f+e^{-1}$, Fatou's lemma applied to the fixed scalar spectral distribution of $x$ yields
\[
 H(x)+e^{-1}\le\liminf_{N\to\infty}\tau(g_N(x))
 \le\liminf_{j\to\infty}\bigl(H(x_j)+e^{-1}\bigr),
\]
which proves \eqref{eq:entropy-lsc}.
\end{proof}

\begin{theorem}[Variational characterization of the logarithmic Sobolev inequality]\label{thm:variational}
Fix $\rho\ge1$. Under the standing assumptions on $(\mathcal M,\tau)$ and $A$, suppose that $\mathcal Q_{\mathrm{sa}}\to L_p(\mathcal M)_{\mathrm{sa}}$ is compact for some $2<p<\infty$. Then the following statements are equivalent:
\begin{equation}\label{eq:variational-LSI}
 \Ent_\tau(x^2)\le2\rho\norm{A^{1/2}x}_2^2
 \qquad\text{for every }x\in\mathcal Q_+;
\end{equation}
\begin{equation}\label{eq:variational-rigidity}
 \rho Ax=x\log x\quad\text{has no solution }x\in\Dom(A)_+
 \text{ with }0<\norm x_2<1.
\end{equation}
\end{theorem}

\begin{proof}
All variations take place in the real Hilbert space $\mathcal Q_{\mathrm{sa}}$. The inclusion into $L_p$ is compact and therefore bounded. Thus there is a constant $C>0$ such that
\begin{equation}\label{eq:variational-Sobolev}
 \norm z_p\le C\norm z_{\mathcal Q}
 =C\bigl(\norm z_2^2+\mathcal E(z)\bigr)^{1/2}.
\end{equation}
Set
\begin{equation}\label{eq:selfadjoint-minimum}
 J_\rho(z)=2\rho\mathcal E(z)-H(z).
\end{equation}
By Lemma~\ref{lem:entropy}, $J_\rho$ is Fr\'echet differentiable on $\mathcal Q_{\mathrm{sa}}$, and $J_\rho=F_\rho$ on the $L_2$ unit sphere. Let
\[
      m=\inf_{\norm{z}_2=1,z\in \mathcal{Q}_{\mathrm{sa}}}J_\rho(z).
\]

We first prove that $m>-\infty$ and that $m$ is attained. Fix $z\in\mathcal Q_{\mathrm{sa}}$ with $\norm z_2=1$. For $s>0$ and $a>0$, the scalar inequality $\log r\le r-1$ gives
\[
s^2\log s^2=\frac{2}{p-2}s^2 \left( \log(\frac{s^{p-2}}{a})+\log a \right)\leq \frac{2}{p-2}s^2\left( \frac{s^{p-2}}{a}-1+\log a \right)=\frac{2}{p-2}\left( \frac{s^p}{a}+(\log a-1)s^2 \right).
\]
The inequality extends to $s=0$ by continuity. Put $a=\norm{z}_p^p>0$. Applying functional calculus to $|z|$ and taking the trace gives
\begin{equation}\label{eq:variational-entropy-bound}
 \begin{aligned}
 H(z)&\le\frac2{p-2}\tau\!\left(\frac{|z|^p}{\norm{z}_p^p}+(\log\norm{z}_p^p-1)|z|^2\right)
 =\frac2{p-2}\log\norm{z}_p^p\\
 &=\frac{p}{p-2}\log\norm z_p^2
 \le\frac{p}{p-2}\log\bigl(C^2(\norm{z}_2^2+\mathcal E(z))\bigr)\\
 &\leq \frac{p}{p-2}\log\left(\frac{C^2p}{p-2}\right)+\frac{p}{p-2}\log\bigl(1+\frac{p-2}{p}\mathcal E(z)\bigr)\\
 &\leq \frac{2p}{p-2}\log\left(\frac{C^2p}{p-2}\right)+\mathcal{E}(z).
 \end{aligned}
\end{equation}
Consequently,
\[
      J_\rho(z)=2\rho\mathcal{E}(z)-H(z)\geq (2\rho-1)\mathcal{E}(z)-\frac{2p}{p-2}\log\left(\frac{C^2p}{p-2}\right).
\]
Thus $J_\rho$ is bounded below on the unit sphere, so $m>-\infty$.

Since $J_\rho(\mathbf1)=0$, we have $m\le0$. The preceding estimate gives
\[
 \{z\in\mathcal Q_{\mathrm{sa}}:\norm z_2=1,\ J_\rho(z)\le0\}
 \subset
 \left\{z\in\mathcal Q_{\mathrm{sa}}:\norm z_2=1,\
 \mathcal E(z)\le
 \frac{2p}{(2\rho-1)(p-2)}
 \log\left(\frac{C^2p}{p-2}\right)\right\}=S.
\]
The set $S$ is bounded in $\mathcal Q_{\mathrm{sa}}$, so the compact inclusion makes it relatively compact in $L_p$. Lower semicontinuity of the closed form and continuity of the $L_2$ norm show that $S$ is closed in $L_p$, hence compact. Equation~\eqref{eq:entropy-remainder} gives continuity of $H$ on $S$, so $J_\rho$ is lower semicontinuous there and attains a minimum. Since $\mathbf1\in S$ and every unit vector outside $S$ has $J_\rho>0$, this minimum equals $m$. By \eqref{eq:modulus-form}, we may therefore choose $x_0\in\mathcal Q_+$ with $\norm{x_0}_2=1$ and $J_\rho(x_0)=m$.

To derive the Lagrange equation, note that the constraint $g(z)=\norm z_2^2-1$ is $C^1$ and regular, since $Dg(x_0)[x_0]=2$. Applying Theorem~\ref{thm:lagrange-multipliers} and reversing the sign of its multiplier gives $\mu\in\mathbb R$ such that
\[
 4\rho\mathcal E(x_0,h)-4\tau((x_0\log x_0)h)
 -2(1+\mu)\tau(x_0h)=0,\qquad h\in\mathcal Q_{\mathrm{sa}}.
\]
Taking $h=x_0$ gives $\mu=m-1$, and hence
\begin{equation}\label{eq:weak-variation}
 \rho\mathcal E(x_0,h)
 =\tau\!\left(\left(x_0\log x_0+\frac m2x_0\right)h\right),
 \qquad h\in\mathcal Q_{\mathrm{sa}}.
\end{equation}
The coefficient on the right belongs to $L_2$ by the growth bound in the proof of Lemma~\ref{lem:entropy}. The characterization of the operator associated with the closed form therefore gives $x_0\in\Dom(A)$ and
\[
      \rho Ax_0=x_0\log x_0+\frac{m}{2}x_0.
\]
Set $x_1=e^{m/2}x_0$. The logarithmic scaling identity yields
\begin{equation}\label{eq:rescaled-Euler}
 \rho Ax_1=x_1\log x_1,\qquad x_1\in\Dom(A)_+,\qquad \norm{x_1}_2=e^{m/2}.
\end{equation}
If \eqref{eq:variational-LSI} fails, homogeneity gives a positive unit vector with $J_\rho<0$. Thus $m<0$, and \eqref{eq:rescaled-Euler} contradicts \eqref{eq:variational-rigidity}. Conversely, pairing a Lagrange solution with itself gives $2\rho\mathcal E(x)=H(x)$, and therefore $F_\rho(x)=\norm x_2^2\log\norm x_2^2<0$ when $0<\norm x_2<1$.
\end{proof}

The following representation applies to a test operator $L=\delta^*\delta$, which may differ from the generator $A$.

\begin{lemma}[Endpoint representation]\label{lem:endpoint}
Let $\mathcal H$ be a Hilbert $\mathcal M$-bimodule with commuting normal unital $*$-representations $\pi_L$ of $\mathcal M$ and $\pi_R$ of $\mathcal M^{\mathrm{op}}$. Let $\delta:\Dom(\delta)\subset L_2(\mathcal M)\to\mathcal H$ be a closed densely defined linear operator, and put $L=\delta^*\delta$. Assume that a unital $*$-subalgebra $\mathcal A_0\subset\mathcal M\cap\Dom(\delta)$ satisfies
\[
 \delta(ab)=\pi_L(a)\delta(b)+\pi_R(b)\delta(a),\qquad a,b\in\mathcal A_0,
\]
and that $(\mathcal A_0)_{\mathrm{sa}}$ is dense in $\Dom(\delta)_{\mathrm{sa}}$ for the norm $(\norm{\cdot}_2^2+\norm{\delta \cdot}_{\mathcal H}^2)^{1/2}$.

For every $x=x^*\in\Dom(L)$, there is a finite positive measure $\mu_x$ on $\sigma(x)^2$ such that
\begin{equation}\label{eq:endpoint}
 \operatorname{Re}\inner{Lx,f(x)}_2
 =\int_{\sigma(x)^2}f^{[1]}(s,t)\,\mathrm d\mu_x(s,t),
 \qquad
 f^{[1]}(s,t)=
 \begin{cases}
 \dfrac{f(s)-f(t)}{s-t},&s\ne t,\\
 f'(s),&s=t.
 \end{cases}
\end{equation}
This holds for every real $C^1$ function $f$ on $\mathbb R$ with bounded derivative, and also for every nondecreasing real $C^1$ function $f$ with $f(x)\in L_2$. For positive $x$, such functions may be defined only on $[0,\infty)$, by extension along the tangent line at zero. In particular,
\begin{equation}\label{eq:massmoment}
 \mu_x(\sigma(x)^2)=\inner{Lx,x}_2,\qquad
 \int_{\sigma(x)^2}(s+t)\,\mathrm d\mu_x(s,t)
 =\operatorname{Re}\inner{Lx,x^2}_2,
\end{equation}
where the second identity requires $x\in L_4$ and its integral is absolutely convergent.
\end{lemma}

\begin{proof}
Fix $x=x^*\in\Dom(L)$. Let $P_x$ be the joint spectral measure of the commuting left and right actions of $x$ on $\mathcal H$, and set
\[
 \mu_x(\Omega)=\inner{\delta x,P_x(\Omega)\delta x}_{\mathcal H}.
\]
This is a positive measure with mass $\norm{\delta x}_{\mathcal H}^2=\inner{Lx,x}_2$, proving the first identity in \eqref{eq:massmoment}.\par

First fix $a\in(\mathcal A_0)_{\mathrm{sa}}$. Repeatedly applying the Leibniz rule gives
\[
      \delta(a^m)=\sum_{k=0}^{m-1}\pi_L(a)^k\pi_R(a)^{m-1-k}\delta a.
\]
By the scalar identity 
\[
      \frac{s^m-t^m}{s-t}=\sum_{k=0}^{m-1}s^kt^{m-1-k},
\]
We have 
\begin{equation}\label{eqn: chain rule}
      \delta p(a)=p^{[1]}(\pi_L(a),\pi_R(a))\delta a
\end{equation}
for every polynomial $p$. By standard approximation procedure, for $x\in \Dom(L)$ and function $f$ which $f'$ is bounded, we have $f(x)\in \Dom(\delta)$ and 
\begin{equation}\label{eq:endpoint-chain-rule}
 \delta f(x)=f^{[1]}(\pi_L(x),\pi_R(x))\delta x
 =\left(\int_{\sigma(x)^2}f^{[1]}(s,t)\,\mathrm dP_x(s,t)\right)\delta x.
\end{equation}
Therefore, we have
\[
 \begin{aligned}
 \operatorname{Re}\inner{Lx,f(x)}_2
 &=\operatorname{Re}\inner{\delta x,\delta f(x)}_{\mathcal H}\\
 &=\inner{\delta x,f^{[1]}(\pi_L(x),\pi_R(x))\delta x}_{\mathcal H}\\
 &=\int_{\sigma(x)^2}f^{[1]}(s,t)\,\mathrm d\mu_x(s,t).
 \end{aligned}
\]
The middle expression is real because $f^{[1]}$ is real-valued.

For nondecreasing $f$ with $f(x)\in L_2$, set
\[
 f_m(s)=f(0)+\int_0^s\min\{f'(r),m\}\,\mathrm dr.
\]
Then $f_m(x)\to f(x)$ in $L_2$ by dominated convergence, since $|f_m(s)|\le|f(0)|+|f(s)|$, and $f_m^{[1]}\uparrow f^{[1]}$. Applying the bounded-derivative case and monotone convergence yields
\[
 \operatorname{Re}\inner{Lx,f(x)}_2
 =\lim_{m\to\infty}\int_{\sigma(x)^2}f_m^{[1]}\,\mathrm d\mu_x
 =\int_{\sigma(x)^2}f^{[1]}\,\mathrm d\mu_x.
\]
This proves \eqref{eq:endpoint} without requiring $f(x)\in\Dom(\delta)$ in the monotone case.

If $x\in L_4$, apply the monotone case to the nondecreasing $C^1$ functions
\[
 g_+(s)=(\max\{s,0\})^2,\qquad g_-(s)=-(\max\{-s,0\})^2.
\]
Their divided differences are nonnegative and integrable, and
\[
 g_+^{[1]}(s,t)-g_-^{[1]}(s,t)=s+t,\qquad
 |s+t|\le g_+^{[1]}(s,t)+g_-^{[1]}(s,t).
\]
Thus $s+t$ is absolutely integrable, and subtracting the two representations gives the second identity in \eqref{eq:massmoment}. The half-line variant follows by extension along the tangent line at zero.
\end{proof}

%% file: Free_group_setting.tex
\section{The free-group setting and reduction}\label{sec:free-group-settings}

Throughout this section, $n\in\mathbb N$ and $\mathcal M=\mathcal L(\mathbb F_n)$, equipped with its canonical tracial state and word-length generator $A\lambda_g=|g|\lambda_g$. The Poisson semigroup $P_t\lambda_g=e^{-t|g|}\lambda_g$ is symmetric, unital, trace-preserving, and Markov; see \cite[Lemmas~1.1--1.2]{MR520930}. The case $n=1$ follows by restriction from $\mathbb F_2$. We verify the hypotheses of Theorem~\ref{thm:variational} and then reduce the LSI to a cubic Fourier estimate. The extension to $\mathbb F_\infty$ is given in Remark~\ref{rem:variational-domain}.

For $k\ge0$, write
\[
 S_k=\{g\in\mathbb F_n:|g|=k\},\qquad
 \Pi_kv=\sum_{g\in S_k}\widehat v(g)\lambda_g.
\]
Shell components of different lengths are orthogonal in $L_2$. An element $v$ is self-adjoint if and only if $\widehat v(g^{-1})=\overline{\widehat v(g)}$, so its shell components are then self-adjoint as well.

\subsection{The word-length derivation}
We use Haagerup's tree construction in \cite[Proof of Lemma~1.2]{MR520930}. Orient the edges of the Cayley tree from $h$ to $hs_j$, where $s_1,\ldots,s_n$ are the generators. Let $\mathsf E_{\mathbb R}$ have orthonormal basis $\varepsilon_{h,j}$ indexed by these edges; reversing an edge changes its sign. Left translation defines $\pi(g)\varepsilon_{h,j}=\varepsilon_{gh,j}$. If $c(g)$ is the signed sum of the edges on the path from $e$ to $g$, cancellation of backtracking edges gives
\begin{equation}\label{eq:cocycle}
 c(gh)=c(g)+\pi(g)c(h),\qquad \norm{c(g)}^2=|g|.
\end{equation}
Following \cite[Example~2.2, pp.~421--422]{MR2470111}, let $\mathsf E_{\mathbb C}$ be the complexification of $\mathsf E_{\mathbb R}$ and equip $\mathcal H=\mathsf E_{\mathbb C}\otimes L_2(\mathcal M,\tau)$ with the actions
\begin{equation}\label{eq:actions}
 \begin{aligned}
 \lambda_g\cdot(\xi\otimes\lambda_h)
     &=\pi(g)\xi\otimes\lambda_{gh},\\
 (\xi\otimes\lambda_h)\cdot\lambda_k
     &=\xi\otimes\lambda_{hk}.
 \end{aligned}
\end{equation}
These actions commute. The right action is amplified right multiplication, while the unitary $W(\xi\otimes\lambda_h)=\pi(h^{-1})\xi\otimes\lambda_h$ conjugates the left action to amplified left multiplication. Thus both actions extend normally to $\mathcal M$, making $\mathcal H$ a Hilbert $\mathcal M$-bimodule.

Define $\delta$ on $\mathbb C[\mathbb F_n]$ by extending $\delta(\lambda_g)=c(g)\otimes\lambda_g$ linearly. The cocycle identity gives
\[
 \delta(vw)=v\cdot\delta w+\delta v\cdot w,
 \qquad v,w\in\mathbb C[\mathbb F_n],
\]
and
\begin{equation}\label{eq:dirichlet}
 \inner{\delta v,\delta w}_{\mathcal H}
   =\sum_g|g|\,\overline{\widehat v(g)}\widehat w(g)
   =\inner{Av,w}_2.
\end{equation}
Orthogonality of the Fourier basis shows that $\delta$ is closable. Denoting its closure by the same symbol, we have
\[
 \Dom(\delta)=
 \left\{x\in L_2(\mathcal M,\tau):
       \sum_{g\in\mathbb F_n}|g|\,|\widehat x(g)|^2<\infty\right\}
 =\Dom(A^{1/2}),\qquad \delta^*\delta=A.
\]
In particular, $\ker\delta=\mathbb C\mathbf1$. This realizes the word-length form in the framework of Cipriani and Sauvageot \cite{MR1986156}. Self-adjoint Fourier polynomials are dense in $\Dom(\delta)_{\mathrm{sa}}$ for the norm $(\norm{\cdot}_2^2+\norm{\delta\cdot}_{\mathcal H}^2)^{1/2}$, so the assumption of Lemma~\ref{lem:endpoint} holds with $\mathcal A_0=\mathbb C[\mathbb F_n]$.

\subsection{Sobolev embeddings}
The following lemma verifies the compact embedding assumption of Theorem~\ref{thm:variational}.
\begin{lemma}\label{lem:compactness}
For every $n\in\mathbb N$ and $2\le p<3$, equip $\mathcal Q=\Dom(A^{1/2})$ with the norm
\[
 \norm x_{\mathcal Q}
 =\left(\norm x_2^2+\norm{A^{1/2}x}_2^2\right)^{1/2}
 =\norm{(1+A)^{1/2}x}_2.
\]
Then the inclusion
\begin{equation}\label{eq:compact-form}
 \mathcal{Q}\longrightarrow L_p(\mathcal M)
\end{equation}
is continuous and compact. Moreover,
\begin{equation}\label{eq:Sobolev-bounds}
 \norm x_4\le C\norm{(1+A)x}_2,\qquad x\in\Dom(A).
\end{equation}
\end{lemma}

\begin{proof}
For a Fourier polynomial $x$, put $v_j=\sum_{2^j\le k<2^{j+1}}\Pi_kx$, $j\ge0$. Haagerup's inequality \cite[Lemma~1.4, pp.~285--286]{MR520930} and Cauchy--Schwarz give
\[
 \begin{aligned}
 \norm{v_j}_\infty
 &\le\sum_{k=2^j}^{2^{j+1}-1}\norm{\Pi_kx}_\infty
 \le\sum_{k=2^j}^{2^{j+1}-1}(k+1)\norm{\Pi_kx}_2\\
 &\le\left(\sum_{k=2^j}^{2^{j+1}-1}(k+1)^2\right)^{1/2}
       \left(\sum_{k=2^j}^{2^{j+1}-1}\norm{\Pi_kx}_2^2\right)^{1/2}\\
 &\le\left(2^j\,2^{2j+2}\right)^{1/2}\norm{v_j}_2
 =2^{1+3j/2}\norm{v_j}_2.
 \end{aligned}
\]
For $2\le p<\infty$, H\"older's inequality therefore yields
\[
 \norm{v_j}_p
 \le\norm{v_j}_\infty^{1-2/p}\norm{v_j}_2^{2/p}
 \le2^{1-2/p}2^{\alpha_pj}\norm{v_j}_2,
 \qquad \alpha_p=\frac32-\frac3p.
\]
Let $s>\alpha_p$ and let $N\ge0$ be an integer. A second application of Cauchy--Schwarz gives
\begin{equation}\label{eq:dyadic-form-tail}
 \begin{aligned}
 \left\|\sum_{j\ge N}v_j\right\|_p
 &\le2^{1-2/p}\sum_{j\ge N}2^{-(s-\alpha_p)j}\,2^{sj}\norm{v_j}_2\\
 &\le2^{1-2/p}
 \left(\sum_{j\ge N}2^{-2(s-\alpha_p)j}\right)^{1/2}
 \left(\sum_{j\ge N}2^{2sj}\norm{v_j}_2^2\right)^{1/2}\\
 &=\frac{2^{1-2/p}\,2^{-(s-\alpha_p)N}}
 {\sqrt{1-2^{-2(s-\alpha_p)}}} \left(\sum_{j\ge N}\sum_{k=2^j}^{2^{j+1}-1}2^{2sj}\norm{\Pi_k x}_2^2\right)^{1/2}\\
 &\le\frac{2^{1-2/p}\,2^{-(s-\alpha_p)N}}
 {\sqrt{1-2^{-2(s-\alpha_p)}}} \left(\sum_{k\geq 2^N}(k^{s}\norm{\Pi_kx}_2)^2\right)^{1/2}\\
 &\le\frac{2^{1-2/p}\,2^{-(s-\alpha_p)N}} {\sqrt{1-2^{-2(s-\alpha_p)}}} \norm{A^sx}_2.
 \end{aligned}
\end{equation}

Cauchy--Schwarz gives $|\tau(x)|\le\norm x_2$ because $\norm{\mathbf1}_2=1$. The spectral weights $(1+k)^{2s}$ dominate both $1$ and $k^{2s}$, so
\[
      \abs{\tau(x)}\leq \norm{x}_2\leq \norm{(1+A)^sx}_2,\qquad \norm{A^sx}_2\leq \norm{(1+A)^sx}_2.
\]
Since $x=\tau(x)\mathbf1+\sum_{j\ge0}v_j$ and $\norm{\mathbf1}_p=1$, the triangle inequality and \eqref{eq:dyadic-form-tail} with $N=0$ give
\[
 \norm x_p \le|\tau(x)|+\frac{2^{1-2/p}}{\sqrt{1-2^{-2(s-\alpha_p)}}}\norm{A^s x}_2
 \le\left(1+\frac{2^{1-2/p}}{\sqrt{1-2^{-2(s-\alpha_p)}}}\right)
       \norm{(1+A)^s x}_2.
\]
Fourier polynomials are dense in $\Dom((1+A)^s)$ for the norm $\norm{(1+A)^s x}_2$, so this estimate extends to a bounded inclusion $I:\Dom((1+A)^s)\to L_p$.

For compactness, put $T_N=\sum_{k<2^N}\Pi_k$. Its range is spanned by the finitely many $\lambda_g$ with $|g|<2^N$, since $n<\infty$. Moreover, $\norm{(1+A)^sT_Nx}_2\le\norm{(1+A)^s x}_2$, so the preceding embedding estimate makes $T_N$ a bounded map into $L_p$. Equation~\eqref{eq:dyadic-form-tail} gives
\[
 \begin{aligned}
 \norm{I-T_N}_{\Dom((1+A)^s)\to L_p}
 &=\sup_{\norm{(1+A)^s x}_2\le1}\norm{(I-T_N)x}_p\\
 &\le\frac{2^{1-2/p}\,2^{-(s-\alpha_p)N}}
 {\sqrt{1-2^{-2(s-\alpha_p)}}}\longrightarrow0.
 \end{aligned}
\]
Thus $T_N$ converges to $I$ uniformly on the domain unit ball. As an operator-norm limit of finite-rank maps, the inclusion is compact.

For $s=1/2$, the domain norm is exactly the form norm:
\[
 \norm{(1+A)^{1/2}x}_2^2=\norm x_2^2+\mathcal E(x)=\norm x_{\mathcal Q}^2.
\]
Since $s-\alpha_p=3/p-1>0$ precisely when $p<3$, this proves \eqref{eq:compact-form}. Taking $s=1$ and $p=4$ gives $s-\alpha_p=1/4$, so \eqref{eq:Sobolev-bounds} holds with $C=1+\sqrt2/\sqrt{1-2^{-1/2}}$.
\end{proof}

\subsection{Reduction to a polynomial inequality}\label{sec:reduction}
We use the functionals $K$ and $G$ from \eqref{eq:KG} to exclude positive solutions of $Ax=x\log x$ with $0<\norm x_2<1$. By Theorem~\ref{thm:variational} with $\rho=1$, this will imply the LSI \eqref{eq:LSI}. The bound \eqref{eq:Sobolev-bounds} gives $x\in L_4$ for every $x\in\Dom(A)$, and hence $x^2,x\log x\in L_2$ when $x\ge0$. Thus the cubic pairing and the functionals $K(x)$ and $G(x)$ are well defined on $\Dom(A)_+$.

\begin{lemma}[The endpoint transform]\label{lem:gap}
For nonconstant $x\in\Dom(A)_+$, put
\[
      z=\frac{\Re\inner{Ax,x^2}_2}{2\inner{Ax,x}_2}.
\]
Then $z>0$ and
\begin{equation}\label{eq:gapbound}
 G(x)\ge2\inner{Ax,x}_2(z-1-\log z).
\end{equation}
Consequently,
\begin{equation}\label{eq:transformed-bound}
 K(x)+G(x)\ge2\inner{Ax,x}_2\left(\frac{\norm{Ax}_2^2}{\inner{Ax,x}_2}-1-\log z\right).
\end{equation}
\end{lemma}

\begin{proof}
Since $\ker\delta=\mathbb C\mathbf1$ and $x$ is nonconstant, $\inner{Ax,x}_2=\norm{\delta x}_{\mathcal H}^2>0$. The Sobolev bound gives $x\in L_4$, so Lemma~\ref{lem:endpoint} applies directly and yields
\begin{equation}\label{eq:cubic-pairing-positive}
 \operatorname{Re}\inner{Ax,x^2}_2
 =\int_{\sigma(x)^2}(s+t)\,\mathrm d\mu_x(s,t)\ge0.
\end{equation}
Thus $z\ge0$. If $z=0$, the nonnegativity of $s$ and $t$ on $\sigma(x)^2$ forces $\mu_x$ to be concentrated at $(0,0)$. Take $g(s)=\sqrt{1+s^2}-1$. This function is $C^1$ with bounded derivative, satisfies $g'(0)=0$, and is strictly increasing on $[0,\infty)$. The chain rule in the proof of Lemma~\ref{lem:endpoint} gives
\[
 \norm{\delta g(x)}_{\mathcal H}^2
 =\int_{\sigma(x)^2}|g^{[1]}(s,t)|^2\,\mathrm d\mu_x(s,t)=0.
\]
Hence $g(x)$ is scalar because $\ker\delta=\mathbb C\mathbf1$. Since $g$ is injective on $[0,\infty)$, $x$ is scalar as well, a contradiction. Therefore $z>0$.

Put $\varphi(s)=s-1-\log s$ and $f(s)=s^2-2s\log s$ for $s>0$, with $f(0)=0$. To apply the endpoint formula at zero, let $f_\varepsilon(s)=f(s+\varepsilon)$ for $\varepsilon>0$. Then $f_\varepsilon$ is $C^1$ on $[0,\infty)$ and $f_\varepsilon'(s)=2\varphi(s+\varepsilon)\ge0$. Since $f_\varepsilon(x)\in L_2$, Lemma~\ref{lem:endpoint} gives
\[
 \operatorname{Re}\inner{Ax,f_\varepsilon(x)}_2
 =2\int_{\sigma(x)^2}\int_0^1\varphi((1-r)s+rt+\varepsilon)\,\mathrm dr\,\mathrm d\mu_x(s,t).
\]
Since $\mu_x\otimes\mathrm dr/\inner{Ax,x}_2$ is a probability measure, Jensen's inequality for the convex function $\varphi$, together with \eqref{eq:massmoment}, yields
\[
 \begin{split}
      \Re\inner{Ax,f_\varepsilon(x)}_2&\ge 2\inner{Ax,x}_2\varphi\left( \frac{1}{\inner{Ax,x}_2}\int_{\sigma(x)^2}\int_0^1\bigl((1-r)s+rt+\varepsilon\bigr)\,\mathrm dr\,\mathrm d\mu_x(s,t) \right)\\
      &=2\inner{Ax,x}_2\varphi\left( \frac{1}{\inner{Ax,x}_2}\int_{\sigma(x)^2}\left( \frac{s+t}{2}+\varepsilon \right)\mathrm d\mu_x(s,t) \right)\\
      &=2\inner{Ax,x}_2\varphi\left( \frac{\Re\inner{Ax,x^2}_2+2\varepsilon\inner{Ax,x}_2}{2\inner{Ax,x}_2}\right)\\
      &=2\inner{Ax,x}_2\varphi(z+\varepsilon).
 \end{split}
\]
For $0<\varepsilon\le1$ and $s\ge0$, we have $|f_\varepsilon(s)|\le C(1+s^2)$. Since $x\in L_4$, dominated convergence gives $f_\varepsilon(x)\to f(x)$ in $L_2$. Letting $\varepsilon\downarrow0$ proves \eqref{eq:gapbound}. Adding $K(x)=2\norm{Ax}_2^2-\operatorname{Re}\inner{Ax,x^2}_2$ and using $\operatorname{Re}\inner{Ax,x^2}_2=2z\inner{Ax,x}_2$ proves \eqref{eq:transformed-bound}.
\end{proof}

For $x\in\Dom(A)_{\mathrm{sa}}$, center $x$ and write its shell components as
\begin{equation}\label{eq:center-notation}
 u=x-\tau(x)\mathbf1,\qquad u_k=\Pi_ku,\qquad b_k=\norm{u_k}_2.
\end{equation}
Since $u$ is centered,
\begin{equation}\label{eq:energies}
 \inner{Au,(A-I)u}_2=\norm{Au}_2^2-\inner{Au,u}_2=\sum_{k\ge1}k(k-1)b_k^2\ge0.
\end{equation}
Since $A\mathbf1=0$ and $\tau(Au)=0$, expanding $x^2=\tau(x)^2\mathbf1+2\tau(x)u+u^2$ gives
\begin{equation}\label{eq:Ccenter}
 \operatorname{Re}\inner{Ax,x^2}_2=2\tau(x)\inner{Au,u}_2+\operatorname{Re}\inner{Au,u^2}_2.
\end{equation}
Moreover, $\inner{Au,u}_2\ge\norm u_2^2>0$ when $x$ is nonconstant.

\begin{lemma}[Length-weighted cubic estimate]\label{lem:weighted}
For every centered self-adjoint $u\in\Dom(A)$,
\begin{equation}\label{eq:weighted}
 \operatorname{Re}\inner{Au,u^2}_2\le2\norm u_2\sqrt{2\inner{Au,u}_2\inner{Au,(A-I)u}_2}
 +\inner{Au,(A-I)u}_2\sqrt{\inner{Au,u}_2-\norm u_2^2}.
\end{equation}
\end{lemma}

The polynomial case is proved in Section~\ref{sec:shell}. By \eqref{eq:Sobolev-bounds}, approximation in the norm $\norm{(1+A)\cdot}_2$ extends the estimate to $\Dom(A)$. We first complete the reduction to this estimate.

\begin{lemma}[The scalar comparison]\label{lem:budget}
Let $x\in\Dom(A)_{\mathrm{sa}}$ be nonconstant and put $u=x-\tau(x)\mathbf1$. Then
\begin{equation}\label{eq:strictbudget}
 z=\frac{\operatorname{Re}\inner{Ax,x^2}_2}{2\inner{Ax,x}_2}
 <\norm x_2\exp\left(\frac{\inner{Au,(A-I)u}_2}{\inner{Au,u}_2}\right).
\end{equation}
\end{lemma}

\begin{proof}
Put $\theta=\inner{Au,(A-I)u}_2/\inner{Au,u}_2\ge0$. Cauchy--Schwarz gives
\[
 \inner{Au,u}_2^2\le\norm u_2^2\norm{Au}_2^2=(1+\theta)\norm u_2^2\inner{Au,u}_2,
\]
so $\inner{Au,u}_2-\norm u_2^2\le\theta\norm u_2^2$. Combining \eqref{eq:Ccenter} and \eqref{eq:weighted} therefore yields
\begin{equation}\label{eq:mean-budget}
 \begin{aligned}
 \frac{\operatorname{Re}\inner{Ax,x^2}_2}{2\inner{Ax,x}_2}
 &=\tau(x)+\frac{\operatorname{Re}\inner{Au,u^2}_2}{2\inner{Au,u}_2}\\
 &\le\tau(x)+\norm u_2\left(\sqrt{2\theta}+\frac{\theta^{3/2}}2\right).
 \end{aligned}
\end{equation}
Since $\tau(x)^2+\norm u_2^2=\norm x_2^2$, another Cauchy--Schwarz inequality gives
\begin{equation}\label{eq:scalarpoly}
 \tau(x)+\norm u_2\left(\sqrt{2\theta}+\frac{\theta^{3/2}}2\right)
 \le\norm x_2\sqrt{1+2\theta+\sqrt2\,\theta^2+\tfrac14\theta^3}.
\end{equation}
For $\theta>0$,
\begin{equation}\label{eq:expstrict}
 e^{2\theta}\ge1+2\theta+2\theta^2+\tfrac43\theta^3
 >1+2\theta+\sqrt2\,\theta^2+\tfrac14\theta^3,
\end{equation}
which proves \eqref{eq:strictbudget}. If $\theta=0$, then \eqref{eq:mean-budget} bounds the ratio by $\tau(x)<\norm x_2$, since $u\ne0$. This proves the remaining case.
\end{proof}
\begin{theorem}[Comparison on the operator domain]\label{thm:domain-comparison}
The expressions $K(x)$ and $G(x)$ in \eqref{eq:KG} are well defined for every $x\in\Dom(A)_+$. If $x$ is nonconstant, then
\begin{equation}\label{eq:scale-comparison}
 K(x)+G(x)>-2\inner{Ax,x}_2\log\norm x_2.
\end{equation}
In particular,
\begin{equation}\label{eq:ball-comparison}
 K(x)+G(x)>0\qquad\text{if }x\text{ is nonconstant and }0<\norm x_2<1.
\end{equation}
\end{theorem}

\begin{proof}
For nonconstant $x$, the word-length spectrum gives $\norm{Ax}_2^2\ge\inner{Ax,x}_2>0$. Write $u=x-\tau(x)\mathbf1$. Combining Lemmas~\ref{lem:gap} and \ref{lem:budget}, using $Ax=Au$ and $\inner{Ax,x}_2=\inner{Au,u}_2$, gives
\[
     \begin{split}
        K(x)+G(x)&\ge2\inner{Ax,x}_2\left(\frac{\norm{Ax}_2^2}{\inner{Ax,x}_2}-1-\log z\right)\\
        &>2\inner{Ax,x}_2\left(\frac{\norm{Ax}_2^2}{\inner{Ax,x}_2}-1-\log\norm{x}_2-\frac{\inner{Au,(A-I)u}_2}{\inner{Au,u}_2}\right)\\
        &=-2\inner{Ax,x}_2\log\norm{x}_2.
     \end{split}
\]
This proves \eqref{eq:scale-comparison} and \eqref{eq:ball-comparison}.
\end{proof}

\begin{proof}[Proof of Theorem~\ref{thm:LSI}]
Fix $2<p<3$. The Markov properties of $P_t$ and Lemma~\ref{lem:compactness} verify the assumptions of Theorem~\ref{thm:variational}. If \eqref{eq:LSI} fails, that theorem with $\rho=1$ gives a solution $x\in\Dom(A)_+$ of $Ax=x\log x$ with $0<\norm x_2<1$. This solution is nonconstant: a nonzero scalar solution $c\mathbf1$ satisfies $c\log c=0$, hence $c=1$, which is incompatible with $\norm x_2<1$. Consequently,
\[
       K(x)+G(x)=2\operatorname{Re}\inner{Ax,Ax-x\log x}_2=0,
\]
contradicting Theorem~\ref{thm:domain-comparison}. Therefore \eqref{eq:LSI} holds.

To see that the coefficient $2$ is optimal, take $h=\lambda_{s_1}+\lambda_{s_1}^*$. Then $h$ is bounded, centered, and self-adjoint, with $Ah=h$. For sufficiently small real $\varepsilon$, the element $x_\varepsilon=\mathbf1+\varepsilon h$ is positive, and Taylor's theorem gives
\[
 \Ent_\tau(x_\varepsilon^2)=2\varepsilon^2\norm h_2^2+O(\varepsilon^3),
 \qquad
 \norm{A^{1/2}x_\varepsilon}_2^2=\varepsilon^2\norm h_2^2.
\]
Their ratio tends to $2$ as $\varepsilon\to0$, proving optimality.
\end{proof}

\begin{remark}[Infinite rank]\label{rem:variational-domain}
The inequality also holds for $\mathbb F_\infty$. The Poisson multipliers remain positive definite because every finite set of words is contained in a finitely generated free subgroup. The derivation and shell estimates apply to finite Fourier sums, with constants independent of the number of generators. Passing to the completion gives the same continuous embeddings and length-weighted cubic estimate. The proofs of Lemmas~\ref{lem:gap} and \ref{lem:budget}, and hence Theorem~\ref{thm:domain-comparison}, then apply unchanged. Let $\mathbb E_j:\mathcal L(\mathbb F_\infty)\to\mathcal L(\mathbb F_j)$ be the conditional expectation onto the subgroup generated by $s_1,\ldots,s_j$. It retains exactly the Fourier coefficients indexed by $\mathbb F_j$, so for $x\in\mathcal Q_+$,
\begin{equation}\label{eq:finite-rank-exhaustion}
 \norm{(1+A)^{1/2}(\mathbb E_jx-x)}_2^2
 =\sum_{g\notin\mathbb F_j}(1+|g|)|\widehat x(g)|^2
 \longrightarrow0.
\end{equation}
Since $\mathbb E_jx\ge0$, finite-rank LSI and entropy lower semicontinuity \eqref{eq:entropy-lsc} give
\[
 \tau(x^2\log x^2)
 \le\liminf_j\tau\bigl((\mathbb E_jx)^2\log((\mathbb E_jx)^2)\bigr)
 \le2\norm{A^{1/2}x}_2^2+\norm x_2^2\log\norm x_2^2.
\]
This proves \eqref{eq:LSI} at infinite rank.
\end{remark}

%% file: Shell_Estimation.tex
\section{The shell projection bound and cubic estimate}\label{sec:shell}

We now prove the Fourier estimate underlying Lemma~\ref{lem:weighted}. For Fourier polynomials $v,w$, the Fourier coefficients of their product satisfy the convolution identity
\begin{equation}\label{eq:convolution}
 \widehat{vw}(g)=\sum_{pq=g}\widehat v(p)\widehat w(q).
\end{equation}
The order of the factors in this identity is essential. In the matrix calculations below, $\operatorname{Tr}$ denotes the ordinary, unnormalized matrix trace, distinct from the trace $\tau$ on $\mathcal M$. For a finite, possibly rectangular matrix $B=(B_{rs})$, the Hilbert--Schmidt and operator norms are given by
\[
 \norm B_{\mathrm{HS}}^2=\operatorname{Tr}(B^*B)=\sum_{r,s}|B_{rs}|^2,\qquad
 \norm B_{\mathrm{op}}=\sup_{\norm\xi_{\ell_2}=1}\norm{B\xi}_{\ell_2},
\]
where $\norm\xi_{\ell_2}=(\sum_s|\xi_s|^2)^{1/2}$ is the Euclidean norm.

The following projection estimate at a fixed cancellation depth is recorded in \cite[Lemma~2.1]{xiezhang2026hypercontractivityfreegroup}, where it is traced to the proof of Haagerup's inequality \cite{MR520930}. We include the proof for completeness.
\begin{lemma}[A fixed cancellation depth]\label{lem:projection}
Suppose that the Fourier polynomials $v_i$ and $w_j$ are supported in shells $i,j\ge0$, respectively. Their product has no component in shell $k$ unless
\begin{equation}\label{eq:triangle}
 |i-j|\le k\le i+j,\qquad i+j+k\text{ is even}.
\end{equation}
For every such $k$,
\begin{equation}\label{eq:projection}
 \norm{\Pi_k(v_iw_j)}_2\le\norm{v_i}_2\norm{w_j}_2.
\end{equation}
Moreover, with $u_k$ and $b_k$ as in \eqref{eq:center-notation}, we have, for $i,j,k\ge0$,
\begin{equation}\label{eq:triple}
 |\tau(u_i u_j u_k)|\begin{cases}
      \le b_i b_j b_k,&\text{if }|i-j|\le k\le i+j\text{ and }i+j+k\text{ is even},\\
      =0,&\text{otherwise}.
 \end{cases}
\end{equation}
\end{lemma}

\begin{proof}
Consider reduced words $p\in S_i$ and $q\in S_j$ in the Fourier supports of $v_i$ and $w_j$, respectively. Reducing the product $pq$ removes a maximal terminal block $c$ of $p$ and $c^{-1}$ from the beginning of $q$. If $d=|c|$, then
\[
 0\le d\le\min(i,j),\qquad k=|pq|=i+j-2d.
\]
This proves the support restriction \eqref{eq:triangle}.

Fix $k$ satisfying \eqref{eq:triangle}, so that the cancellation depth is $d=(i+j-k)/2$. Choose an ordering of each shell and form matrices $U\in M_{|S_{i-d}|\times|S_d|}(\mathbb C)$ and $V\in M_{|S_d|\times|S_{j-d}|}(\mathbb C)$ from the Fourier coefficients of $v_i$ and $w_j$. The rows and columns of $U$ are indexed by $\alpha\in S_{i-d}$ and $c\in S_d$, respectively; those of $V$ are indexed by the same $c\in S_d$ and by $\beta\in S_{j-d}$. Set
\[
 U_{\alpha,c}=
 \begin{cases}
 \widehat v_i(\alpha c),&\alpha c\text{ reduced},\\
 0,&\text{otherwise},
 \end{cases}
 \qquad V_{c,\beta}=
 \begin{cases}
 \widehat w_j(c^{-1}\beta),&c^{-1}\beta\text{ reduced},\\
 0,&\text{otherwise}.
 \end{cases}
\]
Each word $p\in S_i$ has a unique representation as a reduced concatenation $\alpha c$ with $|\alpha|=i-d$ and $|c|=d$. Similarly, each $q\in S_j$ has a unique representation as a reduced concatenation $c^{-1}\beta$ with $|c|=d$ and $|\beta|=j-d$. Thus each input word determines exactly one matrix entry. All entries corresponding to nonreduced concatenations are zero, so we may restrict the sums to reduced concatenations. Parseval's identity gives
\[
 \begin{aligned}
 \norm U_{\mathrm{HS}}^2
 &=\sum_{\substack{\alpha\in S_{i-d}\\c\in S_d}}|U_{\alpha,c}|^2=\sum_{\substack{\alpha\in S_{i-d},\,c\in S_d\\\alpha c\text{ reduced}}}\abs{\widehat v_i(\alpha c)}^2=\sum_{p\in S_i}|\widehat v_i(p)|^2=\norm{v_i}_2^2,\\
 \norm V_{\mathrm{HS}}^2
 &=\sum_{\substack{c\in S_d\\\beta\in S_{j-d}}}|V_{c,\beta}|^2=\sum_{\substack{c\in S_d,\,\beta\in S_{j-d}\\c^{-1}\beta\text{ reduced}}}|\widehat w_j(c^{-1}\beta)|^2=\sum_{q\in S_j}|\widehat w_j(q)|^2=\norm{w_j}_2^2.
 \end{aligned}
\]
For $\alpha\in S_{i-d}$ and $\beta\in S_{j-d}$ with $\alpha\beta$ reduced, every factorization $pq=\alpha\beta$ with $p\in S_i$ and $q\in S_j$ has cancellation depth $d$ and takes the form $p=\alpha c$, $q=c^{-1}\beta$. Thus matrix multiplication sums over exactly the possible cancelled blocks, and \eqref{eq:convolution} gives
\[
 (UV)_{\alpha,\beta}=\sum_{c\in S_d}U_{\alpha,c}V_{c,\beta}=\sum_{\substack{c\in S_d\\\alpha c\text{ reduced}\\c^{-1}\beta\text{ reduced}}}\widehat v_i(\alpha c)\widehat w_j(c^{-1}\beta)
 =\widehat{v_iw_j}(\alpha\beta).
\]
Every reduced concatenation $\alpha\beta$ has length $i+j-2d=k$, and every word in $S_k$ has exactly one such split after $i-d$ letters. Consequently,
\[
 \norm{\Pi_k(v_iw_j)}_2^2=\sum_{\substack{\alpha\in S_{i-d},\,\beta\in S_{j-d}\\\alpha\beta\text{ reduced}}}
       \left|\sum_{c\in S_d}U_{\alpha,c}V_{c,\beta}\right|^2.
\]
Applying Cauchy--Schwarz and then dropping the restriction that $\alpha\beta$ is reduced, which only adds nonnegative terms, gives
\[
 \norm{\Pi_k(v_iw_j)}_2^2\le\sum_{\substack{\alpha\in S_{i-d}\\\beta\in S_{j-d}}}
       \left(\sum_{c\in S_d}|U_{\alpha,c}|^2\right)
       \left(\sum_{c\in S_d}|V_{c,\beta}|^2\right)=\norm U_{\mathrm{HS}}^2\norm V_{\mathrm{HS}}^2=\norm{v_i}_2^2\norm{w_j}_2^2.
\]
This proves \eqref{eq:projection}.

To prove \eqref{eq:triple}, suppose first that $|i-j|\le k\le i+j$ and $i+j+k$ is even. Self-adjointness and shell orthogonality give
\[
      \abs{\tau(u_i u_j u_k)}=\abs{\inner{u_k,u_i u_j}_2}=\abs{\inner{u_k,\Pi_k(u_iu_j)}_2}\leq \norm{u_k}_2\norm{\Pi_k(u_iu_j)}_2\leq b_ib_jb_k.
\]
Otherwise, $\Pi_k(u_iu_j)=0$, so $\tau(u_i u_j u_k)=0$.
\end{proof}

\begin{proof}[Proof of Lemma~\ref{lem:weighted}]
First suppose that $u$ is a centered self-adjoint Fourier polynomial, so all Fourier sums are finite. Averaging over the three cyclic permutations gives
\begin{equation}\label{eq:cyclic}
 \Re\inner{Au,u^2}_2=\Re\sum_{i,j,k\ge1}\frac{i+j+k}{3}\tau(u_i u_j u_k).
\end{equation}
By \eqref{eq:triple}, only triples satisfying \eqref{eq:triangle} contribute to this sum. Every such triple can be written uniquely as
\[
 i=r+s,\qquad j=s+t,\qquad k=t+r,\qquad r,s,t\ge0.
\]
For example, $r=(i+k-j)/2$, with analogous formulas for $s$ and $t$. The triangle and parity conditions ensure that $r,s,t$ are nonnegative integers. The weight in \eqref{eq:cyclic} becomes $2(r+s+t)/3$. Since $u_0=0$, we may include triples with a zero shell index and sum over all $r,s,t\ge0$:
\[
\Re\inner{Au,u^2}_2=\Re\sum_{r,s,t\ge 0}\frac{2(r+s+t)}{3}\tau(u_{r+s} u_{s+t} u_{t+r}).
\]
We distinguish three cases:
\begin{enumerate}
      \item If at least two of $r,s,t$ are zero, then $\tau(u_{r+s}u_{s+t}u_{t+r})=0$ because $u_0=0$.
      \item If exactly one of $r,s,t$ is zero, there are three possible positions for the zero.
      \item If $r,s,t$ are all positive, the scalar majorant is symmetric in $r,s,t$.
\end{enumerate}
Bounding each real part by the corresponding absolute value, applying \eqref{eq:triple}, and using these symmetries, we obtain
\begin{equation}\label{eq:splitmajorant}
 \Re\inner{Au,u^2}_2\le2\sum_{r,s\ge1}(r+s)b_r b_s b_{r+s}+2\sum_{r,s,t\ge1}r\,b_{r+s}b_{s+t}b_{t+r}=4\sum_{r,s\ge1}rb_r b_s b_{r+s}+2\sum_{r,s,t\ge1}r\,b_{r+s}b_{s+t}b_{t+r}.
\end{equation}

For the first term in the final expression, Cauchy--Schwarz yields
\[
4\sum_{r,s\ge1}r b_r b_s b_{r+s}\le 4\left(\sum_{r,s\ge1}r b_r^2b_s^2\right)^{1/2}
        \left(\sum_{r,s\ge1}r b_{r+s}^2\right)^{1/2}.
\]
The first sum under a square root is $\norm u_2^2\inner{Au,u}_2$, while the second is
\[
 \sum_{r,s\ge1}r b_{r+s}^2
 =\sum_{k\ge2}\left(\sum_{r=1}^{k-1}r\right)b_k^2
 =\frac{1}{2}\sum_{k\ge2}k(k-1)b_k^2=\frac{1}{2}\inner{Au,(A-I)u}_2.
\]
Consequently,
\begin{equation}\label{eq:boundary}
      \begin{split}
4\sum_{r,s\ge1}r b_r b_s b_{r+s}
 &\le 4(\norm{u}_2^2\inner{Au,u}_2)^{1/2}\left( \frac{\inner{Au,(A-I)u}_2}{2} \right)^{\frac{1}{2}}=2\norm{u}_2\sqrt{2\inner{Au,u}_2\inner{Au,(A-I)u}_2}.
      \end{split}
\end{equation}

For the sum with $r,s,t\ge1$, choose $N\ge1$ so that $b_k=0$ for $k>N$, and define the real symmetric matrix $M$ and diagonal weight $D$ by
\[
 M=(b_{r+s})_{1\le r,s\le N},\qquad
 D=\operatorname{diag}(1,2,\ldots,N).
\]
Here $M$ is formed from the scalar shell norms $b_k$, and $D$ records the weight $r$. Expanding the trace gives
\[
 \operatorname{Tr}(DM^3)
 =\sum_{r,s,t=1}^N r\,M_{rs}M_{st}M_{tr}
 =\sum_{r,s,t\ge1}r\,b_{r+s}b_{s+t}b_{t+r}.
\]
Thus the second term on the right of \eqref{eq:splitmajorant} is $2\operatorname{Tr}(DM^3)$.

Writing $\lambda_1,\ldots,\lambda_N$ for the real eigenvalues of $M$, we have
\[
 \norm M_{\mathrm{op}}=\max_a|\lambda_a|
 \le\left(\sum_a|\lambda_a|^2\right)^{1/2}
 =\norm M_{\mathrm{HS}}.
\]
Since $M$ is self-adjoint,
\[
      M\leq \norm{M}_{\mathrm{op}}I.
\]
The matrix $\abs M=(M^2)^{1/2}$ commutes with $M$, so multiplication on the left and right by $\abs M$ yields
\[
 M^3=\abs{M}M\abs{M}\le\norm M_{\mathrm{op}}\abs{M}^2=\norm M_{\mathrm{op}}M^2.
\]
Multiplying on the left and right by $D^{1/2}$ and taking traces gives
\[
2\operatorname{Tr}(DM^3)=2\operatorname{Tr}(D^{\frac{1}{2}}M^3D^{\frac{1}{2}})\le2\norm M_{\mathrm{op}}\operatorname{Tr}(D^{\frac{1}{2}}M^2D^{\frac{1}{2}})\le2\norm M_{\mathrm{HS}}\operatorname{Tr}(DM^2).
\]    
The required matrix counts are
\[
\norm M_{\mathrm{HS}}^2=\sum_{r,s=1}^N b_{r+s}^2
 =\sum_{k\ge2}(k-1)b_k^2=\inner{Au,u}_2-\norm{u}_2^2,
\]
and 
\[
      \operatorname{Tr}(DM^2)=\sum_{r,s=1}^N r\,b_{r+s}^2
 =\sum_{k\ge2}\frac{k(k-1)}2b_k^2=\frac{\inner{Au,(A-I)u}_2}{2}.
\]
For $2\le k\le N$, there are $k-1$ positive integer pairs with $r+s=k$, and their weights $r$ sum to $k(k-1)/2$. Terms with $k>N$ vanish because $b_k=0$. Substitution yields
\begin{equation}\label{eq:interior}
       2\operatorname{Tr}(DM^3)\leq \inner{Au,(A-I)u}_2\sqrt{\inner{Au,u}_2-\norm{u}_2^2}.
\end{equation}
Combining \eqref{eq:splitmajorant}, \eqref{eq:boundary}, and \eqref{eq:interior} proves \eqref{eq:weighted} for Fourier polynomials. The approximation following the statement of Lemma~\ref{lem:weighted} extends this estimate to every centered $u\in\Dom(A)_{\mathrm{sa}}$.
\end{proof}